\documentclass[twoside,a4,11pt]{amsart}

\usepackage[english]{babel}
\usepackage{hyperref}
\usepackage{cleveref}
\usepackage[utf8x]{inputenc}
\usepackage{amsmath}
\usepackage{graphicx, stmaryrd}
\usepackage[colorinlistoftodos]{todonotes}
\usepackage{amsmath,amscd,amsthm}
\usepackage{amstext, amsfonts, a4}
\usepackage{amssymb}
\usepackage{tikz-cd}
\usepackage{fancyhdr}
\usepackage{enumerate}

\numberwithin{equation}{section}
\newtheorem{thm}[equation]{Theorem}
\newtheorem{defn}[equation]{Definition}

\newtheorem{prop}[equation]{Proposition}

\newtheorem{lem}[equation]{Lemma}

\theoremstyle{definition}
\newtheorem{ex}[equation]{Example}

\newtheorem{rem}[equation]{Remark}

\renewcommand{\dim}{\operatorname{dim}}
\renewcommand{\deg}{\operatorname{deg}}
\newcommand\ind{\operatorname{{ind}}}

\newcommand\End{\operatorname{End}}
\newcommand\Br{\operatorname{{Br}}}

\newcommand\Int{\operatorname{Int}}

\newcommand\Trd{\operatorname{{Trd}}}

\newcommand\cchar{\operatorname{char}}
\newcommand\disc{\operatorname{{disc}}}

\newcommand{\can}{\operatorname{{can}}}

\DeclareMathOperator{\Tr}{Tr}

\newcommand{\C}{\mathcal{C}}

\renewcommand{\leq}{\leqslant}
\renewcommand{\geq}{\geqslant}

\newcommand\Sym{\operatorname{Sym}}
\newcommand\Skew{\operatorname{Skew}}
\newcommand\Alt{\operatorname{Alt}}
\newcommand\Ad{\operatorname{Ad}}
\newcommand\ad{\operatorname{ad}}

\newcommand{\I}{{I}}
\newcommand{\W}{{W}}

\newcommand{\D}{\mathsf{D}}

\begin{document}
\title{Triality over fields with $\I^3=0$}

\date{\today}
	
\author{Fatma Kader B\.{i}ng\"{o}l}
\author{Anne Qu\'{e}guiner-Mathieu}

\address{Universiteit Antwerpen, Departement Wiskunde, Middelheim\-laan~1, 2020 Ant\-werpen, Belgium.}
\email{FatmaKader.Bingol@uantwerpen.be}
\address{LAGA - CNRS UMR 7539 - Institut Galil\'ee, Universit\'e Sorbonne Paris-Nord, 93430 Villetaneuse, France.}
\email{queguin@math.univ-paris13.fr}

\begin{abstract}
We characterize isotropic trialitarian triples in terms of the Schur indices of the underlying algebras
over a base field $F$ of arbitrary characteristic satisfying $I_q^3 F=0$. 
We also construct anisotropic trialitarian triples over such fields.
	
\medskip\noindent
{\sc Classification} (MSC 2020): 16K20, 16W10, 11E39
	
\medskip\noindent
{\sc{Keywords:}} quadratic pair, triality, isotropy, fields with $\I^3=0$
\end{abstract}

\maketitle

\section{Introduction}
Classical algebraic groups can be described explicitly in terms of central simple algebras endowed with involutions of various types, except for groups of type $\D$ over a field of characteristic $2$. 
In this case, the involution must be replaced by a quadratic pair, consisting of an involution and a linear form defined on the set of symmetric elements. 
Due to the fact that the Dynkin diagram $\D_4$ has more symmetries than other Dynkin diagrams of type $\D$, algebras of degree $8$ with quadratic pairs of trivial discriminant naturally come in the form of triples, called trialitarian triples, see~\cite[\S 42.A]{BOI}, \cite{DQM2021}, and~\S\ref{prel.sec}. 
The relation between the three algebras with quadratic pairs involved in such a triple is provided by the Clifford algebra. More precisely, the Clifford algebra of any element of the triple is the direct product of the other two.
Properties of trialitarian triples are related to the properties of the underlying algebraic groups. 
For example, a triple is isotropic if and only if the associated algebraic groups are isotropic.

The correspondance between classical algebraic groups and algebras with involution, or algebras with quadratic pairs, has proved useful to address the classification problem over some specific base fields. 
Perfect fields of cohomological dimension at most $2$ are of particular interests : Serre's conjecture II states that, over such a field $F$, any $F$-torsor under a semi-simple and simply-connected linear algebraic group is trivial, or, equivalently, has a rational point. 
This conjecture has been extended by Serre to the imperfect case, see P. Gille's lecture note~\cite{Gille} for a detailed account on this problem. 
Higher versions of a weaker statement have also been recently considered in~\cite{ILA}, where rational points are replaced by cycles.

For classical groups, Serre's conjecture II was proved by Bayer-Fluckiger and Parimala when the base field has characteristic different from $2$; Berhuy, Frings and Tignol extended the result to fields of arbitrary characteristic, see~\cite{BP1995}, \cite{BFT2007}. 
In both cases, the authors establish classification theorems for hermitian forms and  algebras with quadratic pairs, in terms of their low-degree cohomological invariants, when the base field has separable cohomological dimension at most $2$. 
Those theorems play a crucial role in the proof of the conjecture. 

As explained in~\cite[Proposition 7.1.7]{Gille}, a field $F$ has $2$-separable cohomological dimension at most $2$ if and only if $I_q^3L=0$ for all finite separable extension $L/F$. 
In particular, the condition $I_q^3F=0$ is weaker than the condition on the separable cohomological dimension, see \cite[Remark 3.6]{EL73}.
Nevertheless, in characteristic different from $2$, the classification theorems in type $\D$ hold as soon as the base field $F$ satisfies $I_q^3F=0$, or equivalently when the reduced norm map is surjective for all quaternion algebras over $F$, see~\cite[Theorem 4.4.1]{BP1995}, \cite[Theorem A]{LT1999}. We do not know whether this weaker condition is enough in characteristic $2$. 

The main result of this paper is \Cref{main.thm}, which provides an explicit illustration of these classification results, under the weaker assumption $I_q^3F=0$. 
Consider a central simple algebra of degree $8$ with a quadratic pair of trivial discriminant; its Clifford invariant is determined by the Brauer classes of the other two central simple algebras involved in the underlying trialitarian triple. 
Assuming the base field $F$ satisfies $I_q^3F=0$, we provide a characterization of isotropy for such a triple in terms of the Schur indices of the three central simple algebras involved. 
We also extend to the characteristic $2$ case the examples of anisotropic trialitarian triples described in~\cite{QMSZ2012}, and prove that they occur over some field $F$ with $I_q^3F=0$. 

An important tool in the proofs is the notion of quadratic extension of an algebra with unitary involution, introduced in~\cite{GQM2009} when the base field has characteristic different from $2$, which we extend to the setting of quadratic pairs in~\S\ref{quadext.sec}. 

\section{Preliminaries and notations}\label{prel.sec}
Our main references are \cite{BOI} for the theory of central simple algebras with involution and quadratic pairs, and \cite{EKM} for the theory of quadratic forms; most of our notations are borrowed from these books. 

Let $F$ be a field of arbitrary characteristic.
As in~\cite{BOI}, we call $(A,\sigma)$ an $F$-algebra with involution if it is either a central simple $F$-algebra with an $F$-linear involution, or a central simple $K$-algebra with a $K/F$-unitary involution, for some étale quadratic $F$-algebra $K$. 
Moreover, following \cite{EKM}, we use the notation $\I_qF$ for the Witt group of nonsingular even-dimensional quadratic forms over $F$, $\I F$ for the ideal of even-dimensional forms in the Witt ring $\W F$ of non-degenerate symmetric bilinear forms over $F$, and $\I^n_qF$ for the product $\I^{n-1}F\cdot\I_qF$ for $n\geq2$. 
Therefore, $I_q^3F=0$ if and only if all $3$-fold quadratic Pfister forms over $F$ are hyperbolic, or, equivalently, the reduced norm map is surjective for all quaternion algebras over $F$. 

Recall from \cite[(5.4)]{BOI} that a \emph{quadratic pair} on the central simple $F$-algebra $A$ is a pair $(\sigma,f)$, where $\sigma$ is an $F$-linear involution on $A$ and $f:\,\Sym(A,\sigma)\rightarrow F$ is a linear map, such that the following two conditions hold: 
\begin{enumerate}
    \item $\dim_F\Sym(A,\sigma)=\frac{n(n+1)}{2}$ and $\Trd_A\bigl(\Skew(A,\sigma)\bigr)=\{0\}$, and
    \item $f\bigl(x+\sigma(x)\bigr)=\Trd_A(x)$ for all $x\in A$,
\end{enumerate}
where $\Sym(A,\sigma)$ and $\Skew(A,\sigma)$ stand for the set of symmetric and skew-symmetric elements of $(A,\sigma)$. 
The map $f$ is called a semi-trace on $(A,\sigma)$. 
If $\cchar F\not =2$, it follows from the definition that $\sigma$ is an orthogonal involution, and the semi-trace satisfies $f(x)=\frac 12 \Trd_A(x)$ for all $x\in\Sym(A,\sigma)$. 
In particular, there is a unique semi-trace associated to a given orthogonal involution, so quadratic pairs and orthogonal involutions are equivalent notions in this case. 
If $\cchar F=2$, then $\sigma$ is symplectic, and several non-isomorphic semi-traces may be defined on $\Sym(A,\sigma)$ in general. Nevertheless, by definition, they all agree on the subvector space $\Alt(A,\sigma)$ of alternating elements in $\Sym(A,\sigma)$. 

Consider an element $\ell \in A$ satisfying $\ell+\sigma(\ell)=1$. 
The $F$-linear map defined by $f_{\ell}(x)=\Trd_A(\ell x)$ for $x\in\Sym(A,\sigma)$ is a semi-trace on $(A,\sigma)$. 
Furthermore, any semi-trace $f$ on $(A,\sigma)$ coincides with $f_{\ell}$ for some $\ell\in A$, which is uniquely determined up to the addition of an alternating element of $(A,\sigma)$, see \cite[(5.7)]{BOI}.

As explained in~\cite[\S 6]{BOI}, a right ideal $I\subset A$ is called isotropic for $(\sigma,f)$ 
if $\sigma(I)I=\{0\}\mbox{ and }f(I\cap\Sym(A,\sigma))=\{0\}.$
We call $(A,\sigma,f)$, or simply $(\sigma,f)$, isotropic if there exists a non-zero isotropic right ideal for $(\sigma,f)$, and hyperbolic if there exists an isotropic right ideal $I$ for $(\sigma,f)$ with $\dim_FI=\frac 12 \dim_FA$. 

We only consider nonsingular quadratic forms. 
Given an even-dimensional vector space $V$ over $F$, there is a bijective correspondance between similarity classes of quadratic forms on $V$ and isomorphism classes of quadratic pairs on $\End_F(V)$, see~\cite[(5.11)]{BOI}. 
We denote by $\Ad(q)=(\End_F(V),\ad_q,f_q)$ the algebra with quadratic pair associated to a quadratic form $q$ under this correspondence. 
Similarly, given a separable quadratic field extension $K/F$, and a hermitian module $(V,h)$ over $(K,\iota)$, where $\iota$ is the unique non-trivial $F$-automorphism of $K$, we let $\Ad(h)$ be the $F$-algebra with $K/F$-unitary involution $(\End_K(V), \ad_h)$. 

To an $F$-algebra of even degree with quadratic pair $(A,\sigma,f)$, one may associate a Clifford algebra $\C(A,\sigma,f)$, which is endowed with a canonical involution $\underline{\sigma}$, and a canonical quadratic pair $(\underline{\sigma},\underline{f})$ if $\deg A\geq 8$, see~\cite[\S 8.B]{BOI} and~\cite[\S 3.1]{DQM2021}. 
When $(A,\sigma,f)=\Ad(q)$, its Clifford algebra is isomorphic to the even part $\C_0(q)$ of the Clifford algebra $\C(q)$ of $q$. 
The center of $\C(A,\sigma,f)$ is an étale quadratic $F$-algebra, which corresponds to a class in $F^\times /F^{\times 2}$ in characteristic different from $2$, and a class in the quotient $F/\wp(F)$ of $F$ by the image of the Artin-Schreier map in characteristic $2$. 
In both cases, this class is called the discriminant of $(\sigma,f)$ and denoted by $\disc(\sigma,f)$. 
In particular, $(\sigma,f)$ has trivial discriminant if and only if the Clifford algebra splits as a direct product $\C^+\times \C^-$ for some central simple $F$-algebras $\C^+$ and $\C^-$. 

A trialitarian triple over $F$ is a triple of $F$-algebras of degree $8$ with quadratic pair $((A,\sigma_A,f_A),(B,\sigma_B,f_B),(C,\sigma_C,f_C))$ such that there exists an isomorphism
$$\varphi_A:(\C(A,\sigma_A,f_A),\underline{\sigma_A},\underline{f_A})\xrightarrow{\sim}(B,\sigma_B,f_B)\times(C,\sigma_C,f_C).$$ 
Moreover, the Clifford algebra of any element in the triple is isomorphic to the direct product of the other two by~\cite[(42.3)]{BOI} and \cite[Theorem 4.11]{DQM2021}: 
\begin{prop}\label{triality-permut}
    Let $\bigl((A,\sigma_A,f_A),(B,\sigma_B,f_B),(C,\sigma_C,f_C)\bigr)$ be a trialitarian triple over $F$.  
    There are isomorphisms 
\begin{equation*}
\begin{split}
    \varphi_B:\bigl(\C(B,\sigma_B),\underline{\sigma_B},\underline{f_B}\bigr)&\xrightarrow{\sim}(C,\sigma_C,f_C)\times(A,\sigma_A,f_A),\\
    \varphi_C:\bigl(\C(C,\sigma_C),\underline{\sigma_C},\underline{f_C}\bigr)&\xrightarrow{\sim}(A,\sigma_A,f_A)\times(B,\sigma_B,f_B).
\end{split}
\end{equation*}
\end{prop}
More precisely, any isomorphism $\varphi_A$ determines isomorphisms $\varphi_B$ and $\varphi_C$ as above, as explained in~\cite[Definition 3.17]{BT2023}. 

\section{Quadratic extensions of algebras with unitary involution}\label{quadext.sec}
Let $(A,\sigma)$ be an $F$-algebra with $F$-linear involution. 
Assume $A$ contains an étale quadratic $F$-algebra $K$, which is $\sigma$-stable. 
Then the centralizer $C$ of $K$ in $A$ is a central simple $K$-algebra of degree $\frac 12 \deg A$ and Brauer equivalent to $A_K=A\otimes_FK$. 
Moreover, $\sigma$ induces an involution $\sigma_C$ on $C$, and we say that $(A,\sigma)$ is a \emph{quadratic extension of $(C,\sigma_C)$}. 
This notion of quadratic extension for algebras with involution was introduced in~\cite[\S 1]{GQM2009} when the base field has characteristic different from $2$. 
It is of particular interest when $\sigma$ acts non-trivially on $K$, so that $\sigma_C$ is unitary. 
Indeed, in this case, any central simple $F$-algebra $A'$, of exponent  $2$ and degree $2\deg C$, and such that $A'_K$ is Brauer-equivalent to $C$, is equipped with a unique symplectic and a unique orthogonal involution that extend the unitary involution $\sigma_C$, see~\cite[Lemma 1.6 \& Proposition 1.9]{GQM2009}. 

The situation is somewhat different in characteristic $2$. Assume $\sigma$ acts non trivially on $K$. 
Pick a generator $u\in K$, with $u^2+u=a\in F$.
Then $\sigma(u)=u+1$, so that $1$ is alternating, and it follows that $\sigma$ is symplectic, see~\cite[(2.6) (2)]{BOI}. 
Moreover, by the same argument as in the proof of ~\cite[Lemma 1.6]{GQM2009}, the involution $\sigma$ is the unique involution on $A$ that extends the unitary involution $\sigma_C$. 
In particular, there is no orthogonal involution on $A$ extending $\sigma_C$ in characteristic $2$. 
In this section, we introduce a notion of quadratic pair extending a unitary involution, which encompasses orthogonal quadratic extensions in characteristic different from $2$.
The definition is motivated by \Cref{split.ex} and \Cref{existeunique.prop} below. 

\begin{defn}\label{quadext.def}
Let $(A,\sigma,f)$ be an $F$-algebra with quadratic pair, and $(C,\sigma_C)$ an $F$-algebra with unitary involution. 
We say that $(A,\sigma,f)$ is a quadratic extension of $(C,\sigma_C)$ if $A$ contains a $\sigma$-stable étale quadratic $F$-algebra $K$ such that :
\begin{enumerate}[$(1)$]
    \item the centralizer of $K$ in $A$, endowed with the involution induced by $\sigma$, is isomorphic to $(C,\sigma_C)$; and
    \item the semi-trace on $\Sym(A,\sigma)$ is given by $f=f_\ell$, where $\ell$ is a generator of $K$ satisfying $\ell^2-\ell\in F$. 
\end{enumerate}
\end{defn}
\begin{rem}\label{unique.rem}
    Since $\sigma_C$ is unitary, it follows from the definition that $\sigma$ acts non-trivially on $K$, hence it coincides with the non-trivial $F$-automorphism $\iota$ of $K$. 
    Therefore, a direct computation shows that for any generator $\ell$ of $K$ satisfying $\ell^2-\ell\in F$, we have $\sigma(\ell)=1-\ell$. 
    So $f_\ell$ is a semi-trace on $A$, see~\cite[(5.7)]{BOI}.  
\end{rem}

\begin{ex}\label{quad-ext-quat}
    Let $Q$ be an $F$-quaternion algebra, with canonical involution $\can_Q$, and let $K$ be an \'etale quadratic $F$-algebra contained in $Q$. 
    Pick a generator $\ell$ of $K$ with $\ell^2-\ell\in F$ and set $\gamma=\Int(2\ell-1)\circ\can_Q$, so that $\gamma=\can_Q$ if $\cchar F=2$, and $\gamma$ is an orthogonal involution when $\cchar F\not =2$. 
    In both cases, $\gamma$ acts non trivially on $K$ and the $F$-algebra with quadratic pair $(Q,\gamma, f_\ell)$ is a quadratic extension of $(K,\iota)$, where $\iota$ is the unique non-trivial $F$-automorphism of $K$. 
    Moreover, one may check that the discriminant of $(\gamma, f_\ell)$ is the class defining the \'etale quadratic $F$-algebra $K$.  

    Moreover, for any $F$-algebra with $F$-linear involution $(C_0,\gamma_0)$, the $F$-algebra with quadratic pair \[(A,\sigma,f)=(C_0,\gamma_0)\otimes_F (Q,\gamma,f_\ell)\] is a quadratic extension of $(C_0,\gamma_0)\otimes_F(K,\iota)$. 
    Indeed, the element $1\otimes \ell\in C_0\otimes Q$ satisfies $\sigma(1\otimes \ell)+1\otimes \ell=1$, so it induces a semi-trace on $\Sym(C_0\otimes Q,\gamma_0\otimes \gamma)$. 
    On an elementary tensor $s\otimes t\in \Sym(C_0,\gamma_0)\otimes \Sym(Q,\gamma)$, it acts as follows : 
    \[f_{1\otimes \ell}(s\otimes t)=\Trd_{C_0\otimes Q}(s\otimes\ell t)=\Trd_{C_0}(s)\Trd_Q(\ell t)=\Trd_{C_0}(s)f_\ell(t).\] 
    Therefore, by~\cite[(5.18)]{BOI}, which defines the tensor product of an algebra with involution and an algebra with quadratic pair, the semi-trace $f_{1\otimes \ell}$ coincides with the semi-trace $f$, as required. 
\end{ex}

\begin{ex}\label{split.ex}
    Let $K$ be an étale quadratic $F$-algebra, $\ell\in K$ a generator with $\ell^2-\ell\in F$, and $\iota$ the non-trivial $F$-automorphism of $K$. 
    Let $V$ be a finite-dimensional $K$-vector space and $h:V\times V\to K$ a nonsingular hermitian form over $(K,\iota)$. 
    We denote by $q_h:\, V\rightarrow F$ the corresponding trace form, that is the quadratic form defined on the $F$-vector space $V$ by $q_h(x)=h(x,x)$ for all $x\in V$, see~\cite[Chapter 10, \S 1]{Schar85}. 
    We claim that $\Ad(q_h)$ is a quadratic extension of $\Ad(h)$.

    This can be checked as follows. Right multiplication by elements of $K$ provides an embedding of $K$ in $\End_F(V)$, with centralizer $\End_K(V)$. 
    Moreover, since \[b_{q_h}(x,y)=h(x,y)+h(y,x)\] for all $x,y\in V$, a direct computation shows that the involution $\ad_{q_h}$ acts as $\ad_h$ on $\End_K(V)$, and as $\iota$ on $K$. 
    This finishes the proof if $\cchar F\neq2$. If $\cchar F=2$, it remains to prove that $f_{q_h}=f_\ell$. 
    Recall that, by~\cite[(5.11)]{BOI}, the semi-trace $f_{q_h}$ is characterized by 
    \[f_{q_h}\bigl(\varphi_{q_h}(v\otimes v)\bigr)=q_h(v)\ \ \forall v\in V,\]
    where $\varphi_{q_h}$ is the isomorphism of algebras with involution \[(V\otimes V, s)\rightarrow (\End_F(V),\ad_{q_h})\] defined by $\varphi_{q_h}(v\otimes w)(x)=vb_{q_h}(w,x)$ for all $v,w,x\in V$, and $s$ is the switch involution, that is $s(v\otimes w)=w\otimes v$. 

    Since the value of a semi-trace is prescribed on alternating elements, it is enough to prove that the linear maps $f_{q_h}$ and $f_\ell$ agree on a complement of $\Alt\bigl(\End_F(V),\ad_{q_h}\bigr)$ in $\Sym\bigl(\End_F(V),\ad_{q_h}\bigr)$. 
    The elements $\varphi_{q_h}(e_i\otimes e_i)$ and $\varphi_{q_h}(e_i\ell\otimes e_i\ell)$ for $1\leq i\leq n$, where $e_1,\dots, e_n$ is a $K$-basis of $V$, provide a basis for such a complement. 
    Therefore, it is enough to prove that for all integer $i$ with $1\leq i\leq n$, we have 
    \[\Tr\bigl(\varphi_{q_h}(e_i\otimes e_i)\ell\bigr)=q_h(e_i)\mbox{ and }\Tr\bigl(\varphi_{q_h}(e_i\ell\otimes e_i\ell)\ell\bigr)=q_h(e_i\ell),\]
    where $\Tr$ is the usual trace on $\End_F(V)$. 
    The map $\varphi_{q_h}(e_i\otimes e_i)\ell$ sends any vector $x\in V$ to $e_ib_{q_h}(e_i,x\ell)$. 
    Therefore, its trace is given by 
\begin{equation*}
\begin{split}
    \Tr\bigl(\varphi_{q_h}(e_i\otimes e_i)\ell\bigr)&=b_{q_h}(e_i,e_i\ell)=h(e_i,e_i\ell)+h(e_i\ell,e_i)\\
    &=h(e_i,e_i)(\ell+\iota(\ell))=q_h(e_i).
\end{split}
\end{equation*}   
    Similarly, the trace of $\varphi_{q_h}(e_i\ell\otimes e_i\ell)\ell$ is equal to 
    \[b_{q_h}(e_i\ell,e_i\ell^2)=h(e_i\ell,e_i\ell)(\ell+\iota(\ell))=q_h(e_i\ell),\]
    and this proves $f_\ell=f_{q_h}$.    
\end{ex}
    
With Definition~\ref{quadext.def} in hand, the existence and uniqueness result for quadratic extensions given in~\cite[\S 1]{GQM2009} extends to an arbitrary base field as follows: 
\begin{prop} \label{existeunique.prop}
    Let $K$ be an étale quadratic $F$-algebra, and $(C,\sigma_C)$ an $F$-algebra with $K/F$-unitary involution. 
    Assume $C$ has exponent at most $2$, and consider a central simple $F$-algebra $A$ of exponent $2$ and degree $2\deg C$ such that $A_K$ is Brauer equivalent to $C$. 
    There exists a unique symplectic involution $\gamma$ and a unique quadratic pair $(\sigma, f)$ on $A$ such that $(A,\gamma)$ and $(A,\sigma,f)$ are quadratic extensions of $(C,\sigma_C)$.
\end{prop} 
\begin{proof}
    In characteristic different from $2$, the result is stated in~\cite[Lemma 1.6 \& Proposition 1.9]{GQM2009}. 
    The same argument, slightly modified, also works in characteristic $2$. 
    The main difference is that the involution $\sigma$ in the quadratic pair is orthogonal in characteristic different from $2$, while it coincides with $\gamma$ in characteristic $2$. We sketch the argument, for further use. 
    
    The starting point is provided by~\cite[pp. 380, 381]{ET2001}. 
    Since $C$ has exponent at most $2$, it admits an involution $\nu$ of orthogonal type. 
    The semi-linear automorphism $\rho=\nu\circ \sigma_C$ satisfies  $\rho^2=\Int(w)$, for some element $w\in C^\times$, which can be chosen such that $\nu(w)=w$ and $\sigma_C(w)=w$.
    Consider the $F$-algebra $C\oplus Cv$ with multiplication defined by $v^2=w$ and $vx=\rho(x)v$ for all $x\in C$. 
    It is a central simple $F$-algebra by~\cite[Chapter 11, Theorem 10]{Alb39}. 
    In addition, it contains a copy of $K$, with centralizer $C$. 
    Therefore, the algebras $C\oplus Cv$ and $A$ are Brauer-equivalent over $K$. Their Brauer classes over $F$ differ by the class of a quaternion algebra which splits over $K$, so it can be written as $(K,\lambda)$ for some $\lambda\in F^\times$.
    By~\cite[(13.41)]{BOI}, replacing $w$ by $w\lambda$ in the construction above, we may assume $A\simeq C\oplus Cv$. 
    
    Let $\sigma$ be the involution on $A$ that acts as $\sigma_C$ on $C$ and with $\sigma(v)=v$. 
    If $\cchar F\not =2$, the involution $\sigma$ has the same type as $\nu$ by the proof of \cite[Proposition 1.9]{GQM2009}, hence it is orthogonal. 
    If $\cchar(F)=2$, since $\sigma$ acts as the non-trivial automorphism on $K$, we have $1\in\Alt(A,\sigma)$ so that $\sigma$ is symplectic. 
    In both cases, $(A,\sigma,f_\ell)$ is a quadratic extension of $(C,\sigma_C)$, where $\ell\in K$ is any generator with $\ell^2-\ell\in F$. 
    Uniqueness of $\sigma$ follows directly from the Skolem-Noether theorem as in the proof of \cite[Lemma 1.6]{GQM2009}. 
    It only remains to prove that $f_\ell$ is unique when $\cchar F=2$. 
    The étale quadratic $F$-algebra $K$ can be described as $K=F(u)$ with $u^2+u=a$ for some $a\in F$. 
    The condition on  $\ell$ guarantees $\ell=u+\mu$ for some $\mu\in F$. 
    Moreover, as $\sigma$ is symplectic, $F$ is contained in $\Alt(A,\sigma)$. 
    Hence, it follows by \cite[(5.7)]{BOI} that $f_\ell=f_u$. 
\end{proof}

The next proposition generalizes \cite[Lemma 4.3]{ET2001}.
\begin{prop}\label{isot-subalg+unitary->quad-pair=isot}
    Let $(A,\sigma,f)$ be a quadratic extension of the $F$-algebra with $K/F$-unitary involution $(C,\sigma_C)$. 
    If $(C,\sigma_C)$ is isotropic (respectively hyperbolic), then $(A,\sigma,f)$ is isotropic (respectively hyperbolic).
\end{prop}
\begin{proof}
    The unitary involution $\sigma_C$ extends to a unique quadratic pair on $A$; therefore, we may use the explicit description of $(A,\sigma,f)$ given in the proof of \Cref{existeunique.prop}. 
    In particular, $A=C\oplus Cv$. Consider a right ideal $J$ of $C$ and let $I=J\oplus Jv$.
    Using the rules defining multiplication in $A=C\oplus Cv$, one may easily check that $I$ is a right ideal of $A$. 
    Moreover, we have that $\dim_FI=2\dim_FJ=4\dim_K J$. 
    Hence $\dim_FI=\frac 12 \dim_F A$ if and only if $\dim_KJ=\frac 12 \dim_K C$, and to conclude the proof, it only remains to prove that $I$ is isotropic for $(\sigma,f)$ when $J$ is isotropic for $\sigma_C$. 
    
    So, let us assume $J$ is an isotropic ideal for $\sigma_C$, and consider an idempotent element $e\in C$ such that $J=eC$. 
    We have $\sigma_C(e)e=0$ and $I=eA$; in particular, $I$ is isotropic for $\sigma$, and this finishes the proof if $\cchar F\neq2$. 
    Assume now $\cchar F=2$; we need to prove that the semi-trace $f=f_u$ vanishes on $I \cap \Sym(A,\sigma)$. 
    
    For all $z=x+yv\in A$, with $x,y\in C$, we have $\sigma(z)=\sigma(x)+v\sigma(y)=\sigma(x)+\nu(y)v.$ 
    Therefore, \[\Sym(A,\sigma)=\Sym(C,\sigma_C)\oplus \Sym(C,\nu)v.\]
    We claim that $f_u$ vanishes on $\Sym(C,\nu)v$. 
    Indeed, since $\nu$ is $K$-linear, it satisfies $\nu(u)=u$. So, for $y\in\Sym(C,\nu)$, we have $\sigma(uyv)=\nu(uy)v=uyv$. 
    Since $\sigma$ is symplectic, it follows by~\cite[(2.6)]{BOI} that $f_u (yv)=\Trd_A(u yv)=0$. 
 
    Consider now an element $x\in \Sym(C,\sigma_C)$ and assume in addition $x\in I$. 
    In particular, it satisfies $x=ex=\sigma(x)=\sigma(x)\sigma(e)$. Since $\sigma(e)e=0$, we get $x^2=0$. 
    It follows that $(u x)^2=0$ and hence $f_u(x)=\Trd_A(u x)=0$. This concludes the proof. 
\end{proof}

By Proposition~\ref{existeunique.prop}, the quadratic pair $(\sigma,f)$ on $A$ is determined by $(C,\sigma_C)$. 
The next result shows that its invariants are determined by those of $(C,\sigma_C)$: 
\begin{lem}
    Let $(A,\sigma,f)$ be an $F$-algebra with quadratic pair and assume that it is a quadratic extension of the $F$-algebra with $K/F$-unitary involution $(C,\sigma_C)$. 
    We assume in addition that $C$ has even degree, so that $A$ has degree divisible by $4$. 
    Then $(\sigma,f)$ has trivial discriminant and its Clifford invariant coincides with the class of the Discriminant algebra ${\mathcal {D}}(C,\sigma_C)$ in $\Br(F)/\langle[A]\rangle$. 
\end{lem}
\begin{proof}
    Let $F_A$ be the function field of the Severi-Brauer variety of $A$, which is a generic splitting field for $A$. 
    The field $F$ is algebraically closed in $F_A$, and the restriction map induces an injection $\Br(F)/\langle[A]\rangle\rightarrow \Br(F_A)$ by Amistur's theorem~\cite[Theorem 5.4.1]{GS2017}. 
    Therefore, it suffices to check the result when $A$ is split.  
    In view of \Cref{split.ex}, we may assume $(C,\sigma_C)$ is $\Ad(h)$ for some nonsingular hermitian form $h$ defined on a finite-dimensional $K$-vector space, and $(A,\sigma,f)$ is $\Ad(q_h)$, where $q_h$ is the trace form of $h$. 
    The result then follows from the description of the invariants of $q_h$ \cite[Chapter 10, Remark 1.4]{Schar85}, and the computation of the discriminant algebra of $\Ad(h)$ \cite[(10.35)]{BOI}. 
\end{proof}

\begin{rem}
    When $C$ has odd degree, so that $A$ has degree $2$ mod $4$, the same argument shows that the discriminant of $(\sigma,f)$ coincides with the class defining the \'etale quadratic extension $K$. 
\end{rem}

\section{Isotropic quadratic pairs and trialitarian triples}
This section aims at proving~\Cref{main.thm}, which characterizes isotropic tria\-litarian triples when the base field satisfies $I_q^3F=0$. 
In particular, over such a field, algebras of index $2$ with anisotropic quadratic pairs with trivial discriminant and Clifford algebra of index at most $2$ can only exists in degree $4$, as we proceed to show. We start with the following preliminary result: 

\begin{prop}\label{iso-quad-pair-I^3=0}
    Assume that $\I_q^3F=0$. Let $(A,\sigma,f)$ be an $F$-algebra with anisotropic quadratic pair. 
    If there exists an étale quadratic $F$-algebra $K$ such that $(A,\sigma,f)_K$ is split and hyperbolic, then $\deg A\leq4$.
\end{prop}
\begin{proof}
    It follows from the assumptions that 
    the algebra $A$ is Brauer equivalent to a quaternion algebra $Q$ containing $K$.
    Moreover, by~\cite[Proposition 7.3]{BD2015}, $(A,\sigma,f)$ decomposes as 
    \[(A,\sigma,f)\simeq\Ad(\varphi)\otimes (Q,\gamma,f_\ell),\] where $\varphi$ is a non-alternating symmetric bilinear form over $F$, and $(Q,\gamma,f_\ell)$ is a quadratic extension of $(K,\iota)$. 
    Therefore, as explained in \Cref{quad-ext-quat}, $(A,\sigma, f)$ is a quadratic extension of the $F$-algebra with unitary involution $(C,\sigma_C)=\Ad(\varphi)\otimes (K,\iota)$. 
    Moreover, since $(A,\sigma,f)$ is anisotropic, $(C,\sigma_C)$ also is anisotropic by~\Cref{isot-subalg+unitary->quad-pair=isot}. 
    We claim that this implies $\deg C\leq 2$ and hence $\deg A\leq 4$ as required.
    
    Indeed, since $(C,\sigma_C)=\Ad(\varphi)\otimes (K,\iota)$, $\sigma_C$ is the adjoint involution with respect to a hermitian form $h$ over $(K,\iota)$, and any diagonalisation of the bilinear form $\varphi$ provides a diagonalisation of $h$.
    Therefore, a direct computation shows the trace form $q_h$ of $h$ is given by $q_h=\varphi\otimes N_{K/F}$, where $N_{K/F}$ is the norm form of the \'etale quadratic $F$-algebra $K$.
    Assume $\dim\varphi=3$; then $q_h$ is a $6$-dimensional subform of a scaled $3$-fold Pfister quadratic form.
    Hence, as $\I_q^3F=0$, $q_h$ and $h$ are isotropic. This shows that if $\deg C\geq3$, then $\sigma_C$ is isotropic, as claimed.  
\end{proof}

As a consequence, we get the following: 
\begin{prop}\label{deg4n-quad-pair-isot-I^3=0}
    Assume that $\I_q^3F=0$. Let $(A,\sigma,f)$ be an $F$-algebra with quadratic pair such that $\deg A$ is divisible by $4$, $\ind A\leq2$ and $\disc(\sigma,f)$ is trivial.
    Let $\C^+$ and $\C^-$ be the simple components of the Clifford algebra $\C(A,\sigma,f)$.
\begin{enumerate}[$(1)$]
    \item If $\deg A>4$ and $\C^+$ and $\C^-$ have index at most $2$, then $(A,\sigma,f)$ is isotropic.
    \item If $\C^+$ or $\C^-$ is split, then $(A,\sigma,f)$ is hyperbolic.
\end{enumerate}    
\end{prop}
\begin{rem} 
    The assertion $(2)$ is valid in arbitrary index under the stronger condition that $F$ has $2$-separable cohomological dimension ${\mathrm {sd}}_2(F)$ at most $2$ by~\cite[Theorem 5.1]{BFT2007}. 
    Their argument provides the outline of an alternative proof for our result. 
    Indeed, their proof is based on a classification theorem for unitary involutions over fields $F$ with ${\mathrm {sd}}_2(F)\leq 2$, proved in~\cite[Theorem 3.2]{BFT2007}. 
    Since we only consider algebras of index $2$, we would need an analogous statement under the weaker assumption that $I_q^3F=0$ only for split algebras with unitary involutions. 
    This is proved in characteristic different from $2$ in~\cite[Lemma 4.1.3]{BP1995}, with an argument which is valid in arbitrary characteristic. 
\end{rem}
\begin{proof} 
    $(1)$ By the fundamental relations \cite[(9.12)]{BOI}, $\C^+\otimes_F\C^-$ is Brauer equivalent to $A$. 
    In addition, all three algebras have index at most $2$. 
    Therefore, there exists a separable quadratic field extension $K/F$ over which they are split, see~\cite[Theorem 98.19]{EKM}.
    It follows that $(A,\sigma,f)_K\simeq\Ad(\varphi)$ for some nonsingular quadratic form $\varphi$ over $K$, of dimension $4n$ for some integer $n$, and with even Clifford algebra $\C_0(\varphi)\simeq \C(A,\sigma,f)_K\simeq \C^+_K\times\C_K^-$. 
    So, the quadratic form $\varphi$ has trivial discriminant and trivial Clifford invariant. 
    Since $\I_q^3F=0$, we have that $\I_q^3K=0$, by \cite[Theorem 34.22]{EKM}.
    Therefore $\varphi$ is hyperbolic~\cite[Theorem 3.11]{EL73}, \cite[Theorem 16.3]{EKM}. 
    Then $(A,\sigma,f)_K$ is split hyperbolic, and the conclusion follows from~\Cref{iso-quad-pair-I^3=0}, since $\deg A>4$.
     
    $(2)$ Assume now that $\C^+$ or $\C^-$ is split. 
    If $A$ is split, then $(\sigma,f)$ is adjoint to a quadratic form $\varphi$ with trivial discriminant and trivial Clifford invariant. 
    Therefore, in this case, it is hyperbolic, since $I_q^3F=0$. 
    So we may assume $A$ has index $2$. Let $(A_0,\sigma_0,f_0)$ be the anisotropic part of $(A,\sigma,f)$, as in~\cite[Proposition 1.11]{BFT2007}.
    The algebra $A_0$ also has degree divisible by $4$, and we claim that $(A_0,\sigma_0,f_0)$ has trivial discriminant, and $\C(A_0,\sigma_0,f_0)$ has a split component. 
    This would follow from the fact that the discriminant and the Clifford invariant are invariants of the Witt class of a quadratic form if $A$ were split, see~\cite[Lemma 13.4 and Lemma 14.2]{EKM}. 
    The general case reduces to this observation by extending scalars to the function field of the Severi-Brauer variety of a division algebra $D$ Brauer equivalent to $A$ and $A_0$. 
    In view of (1), it follows that $A_0$ has degree $0$ or $4$, and we claim that the degree $4$ case is impossible. 
    Indeed, by the exceptional isomorphism ${\mathsf{A}}_1^2\equiv{\mathsf{D}}_2$, see~\cite[\S 15.B]{BOI}, if $A_0$ has degree $4$, and $(\sigma_0,f_0)$ is anisotropic with trivial discriminant, then both components of its Clifford algebra are non split. 
    More precisely, $(A_0,\sigma_0,f_0)$ is the norm of its Clifford algebra, viewed as an algebra with quadratic pair.
    Since $(\sigma_0,f_0)$ has trivial discriminant, by \cite[(15.12)]{BOI}, this means
    \[(A_0,\sigma_0,f_0)\simeq(Q^+\otimes Q^-, \gamma^+\otimes\gamma^-,f_\otimes),\] 
    where $Q^+$ and $Q^-$ are $F$-quaternion algebras isomorphic to the two components of the Clifford algebra of $(A_0,\sigma_0,f_0)$, hence Brauer-equivalent to $\C^+$ and $\C^-$, $\gamma^+$, $\gamma^-$ denote the canonical involution of $Q^+$, $Q^-$ respectively, and $f_\otimes$ is the canonical semi-trace on this tensor product defined in~\cite[(5.21)]{BOI}. 
    In particular, $(\sigma_0,f_0)$ is anisotropic if and only if both $Q^+$ and $Q^-$ are non-split, see~\cite[(15.14)]{BOI}.
\end{proof}

Let $T=\bigl((A,\sigma_A,f_A),(B,\sigma_B,f_B),(C,\sigma_C,f_C)\bigr)$ be a trialitarian triple over $F$, see~\S\ref{prel.sec} for a definition. 
The triple of natural numbers $(a,b,c)$ obtained by permutation of the entries of $(\ind A,\ind B,\ind C)$ and such that $a\leq b\leq c$ will be called the \emph{index of the triple $T$}. We denote it by $\ind T$.
The triple $T$ is called isotropic if one of the three algebras with quadratic pair involved in $T$ is isotropic. 
This holds if and only if all three algebras with quadratic pairs are isotropic, since they correspond to the same algebraic group (see also~\cite{Ga1999} for a direct proof when $\cchar F\not =2$). 

Isotropy for such a triple can only occur for some specific values of the index of $T$, as the following lemma shows:
\begin{lem} \label{ind.lem}
    If $T$ is an isotropic trialitarian triple over $F$, then its index belongs to 
    $\{(1,1,1), (1,2,2), (1,4,4), (2,2,2)\}$. 
\end{lem}
\begin{proof}
    Let $T=\bigl((A,\sigma_A,f_A),(B,\sigma_B,f_B),(C,\sigma_C,f_C)\bigr)$ be an isotropic trialitarian triple over $F$. 
    Then all three algebras with quadratic pair are isotropic, hence the triple $T$ does not contain any division algebra. 
    In particular, $\ind A$, $\ind B$ and $\ind C$ belong to $\{1,2,4\}$. 
    Permuting the entries if necessary, we may assume that the triple is ordered by indices, i.e. $\ind A\leq \ind B\leq \ind C$. 
    If $\ind A=1$, then the two components of $\C(A,\sigma_A,f_A)$ are isomorphic, see~\cite[(9.12)]{BOI}. 
    Therefore, in this case $\ind T\in\{(1,1,1), (1,2,2), (1,4,4)\}$. 
    Assume now $\ind C=4$. 
    Hence all non trivial proper right ideals in $C$ have dimension $\frac 12 \dim_FC$. 
    Since in addition $(C,\sigma_C,f_C)$ is isotropic, it is hyperbolic. 
    So its Clifford algebra has a split component, see~\cite[(8.31)]{BOI}. Hence $\ind T=(1,4,4)$ in this case. 
    The only remaining possible value is $(2,2,2)$ and this concludes the proof. 
\end{proof}

The converse also holds under some condition on the base field, as we proceed to show; more precisely, we have the following:  
\begin{thm}\label{main.thm}
    Assume that $\I_q^3F=0$ and let $T$ be a trialitarian triple over $F$. 
    The following are equivalent: 
\begin{enumerate}[$(1)$]
    \item The triple $T$ is isotropic.
    \item $\ind T\in\{(1,1,1), (1,2,2), (1,4,4), (2,2,2)\}$.
\end{enumerate}	
\end{thm}
\begin{proof}
    The implication $(1)\Rightarrow(2)$ is given by \Cref{ind.lem}, and it only remains to prove the converse. 
    Let $T$ be a trilitarian triple with index as in (2). If $\ind T=(2,2,2)$, $T$ is isotropic by \Cref{deg4n-quad-pair-isot-I^3=0}.
    Otherwise, one element in the triple is isomorphic to $\Ad(q)$ for some nonsingular quadratic form $q$ of dimension $8$, and the other two algebras are Brauer equivalent to a biquaternion algebra $B$, of index $1$, $2$ or $4$.
    By \cite[\S 16.A]{BOI}, $B$ is also Brauer equivalent to $\C(q')$ for a nonsingular quadratic form $q'$ of dimension $6$ with trivial discriminant. 
    Therefore, the Witt class of $q\perp-q'$ belongs to $I_q^3F$, which is trivial by our assumption, hence $q\perp-q'$ is hyperbolic. 
    Since $\dim q>\dim q'$, it follows that $q$ is isotropic. Hence $\Ad(q)$ is isotropic, whereby $T$ is isotropic.
\end{proof}

\begin{rem}
    Assume $I_q^3F=0$. Let $T=\bigl((A,\sigma_A,f_A),(B,\sigma_B,f_B),(C,\sigma_C,f_C)\bigr)$ be an isotropic trialitarian triple over $F$. 
    We assume $\ind A\leq\ind B\leq\ind C$.
    It follows from the proof of \Cref{main.thm} that the index of $(A,\sigma_A,f_A)$ in the sense of~\cite[\S 6.A]{BOI}, and hence the Tits index of the underlying algebraic groups, are determined by the index of $T$. 
    The correspondence is given in the table below, cf.~\cite{Tits} and~\cite[\S\,II.5]{DCGT}.
\begin{table}[hbt]

\label{table}
\begin{tabular}{ccc}
    Index of $T$ & Index of $(A,\sigma_A,f_A)$ & Tits index of ${\mathrm{Spin}}(A,\sigma_A,f_A)$  \\
\hline
&&\\
    (1,1,1)  &\{0,1,2,3,4\}&$\begin{picture}(7,2)
    \put(0,2){\circle*{3}}
    \put(0,2){\circle{7}}
    \put(17,2){\circle*{3}}
    \put(17,2){\circle{7}}
    \put(34,14.9){\circle*{3}}
    \put(34,14.9){\circle{7}}
    \put(34,-10.9){\circle*{3}}
    \put(34,-10.9){\circle{7}}
    \put(0,2){\line(2,0){17}}
    \put(17,2){\line(4,3){17}}
    \put(17,2){\line(4,-3){17}}
\end{picture}$\\
&&\\
&&\\
    (1,2,2)&\{0,1,2\}& $\begin{picture}(7,2)
    \put(0,2){\circle*{3}}
    \put(0,2){\circle{7}}
    \put(17,2){\circle*{3}}
    \put(17,2){\circle{7}}
    \put(34,14.9){\circle*{3}}
    \put(34,-10.9){\circle*{3}}
    \put(0,2){\line(2,0){17}}
    \put(17,2){\line(4,3){17}}
    \put(17,2){\line(4,-3){17}}
\end{picture}$\\
&&\\
&&\\
    (1,4,4)&\{0,1\}& $\begin{picture}(7,2)
    \put(0,2){\circle*{3}}
    \put(0,2){\circle{7}}
    \put(17,2){\circle*{3}}
    \put(34,14.9){\circle*{3}}
    \put(34,-10.9){\circle*{3}}
    \put(0,2){\line(2,0){17}}
    \put(17,2){\line(4,3){17}}
    \put(17,2){\line(4,-3){17}}
\end{picture}$ \\
&&\\
&&\\
    (2,2,2)&\{0,2\}&$\begin{picture}(7,2)
    \put(0,2){\circle*{3}}
    \put(17,2){\circle*{3}}
    \put(17,2){\circle{7}}
    \put(34,14.9){\circle*{3}}
    \put(34,-10.9){\circle*{3}}
    \put(0,2){\line(2,0){17}}
    \put(17,2){\line(4,3){17}}
    \put(17,2){\line(4,-3){17}}
\end{picture}$\\
\end{tabular}

\end{table}

    Note that in the second and third cases of the table, the diagram is not invariant under rotation. 
    This reflects the fact that $(B,\sigma_B,f_B)$ and $(C,\sigma_C,f_C)$, which are isomorphic as algebras with involution when $A$ is split, do not have the same index as $(A,\sigma_A,f_A)$ in these two cases. More precisely, they have index $\{0,2,4\}$ (resp. $\{0,4\}$) when the triple $T$ has index $(1,2,2)$ (respectively $(1,4,4)$).
\end{rem}

A trialitarian triple that contains a division algebra is anisotropic. The converse does not hold in general, even over fields with $I_q^3F=0$. 
Explicit examples of anisotropic triples of index bounded by $4$ are provided in~\cite{QMSZ2012} in characteristic different from $2$. Using orthogonal sum of quadratic pairs, we now extend these examples to the characteristic $2$ case for fields with $I_q^3F=0$.

Let $(A,\sigma,f), (A_1,\sigma_1,f_1)$ and $(A_2, \sigma_2,f_2)$ be $F$-algebras with quadratic pair. 
Recall that $(A,\sigma,f)$ is called an orthogonal sum of $(A_1,\sigma_1,f_1)$ and $(A_2,\sigma_2,f_2)$, and we write
$$(A,\sigma,f)\in(A_1,\sigma_1,f_1)\boxplus(A_2,\sigma_2,f_2)$$
if there are symmetric orthogonal idempotents $e_1$ and $e_2$ in $A$ with $e_1+e_2=1$ such that for $i=1,2$
$$(e_iAe_i,\sigma|_{e_iAe_i})\simeq(A_i,\sigma_i),$$
so that we may identify $A_i$ with a subset of $A$, and
$$f(x_i)=f_i(x_i)\quad \text{for all}\quad x_i\in\Sym(A_i,\sigma_i).$$
Note that, in general, the $F$-algebra with quadratic pair $(A,\sigma,f)$ is not uniquely determined by the two summands $(A_i,\sigma_i,f_i)$ for $i=1,2$.

Given $a\in F$ with $-4a\neq1$ and $b\in F^{\times}$, we use the notation $[a,b)_F$ in the following examples for the $F$-quaternion algebra \[F\oplus Fi\oplus Fj\oplus Fij,\mbox{ where } i^2-i=a,\ j^2=b\mbox{ and }ji=(1-i)j.\]
In particular, if $F$ has characteristic $2$, then $F(i)$ is a quadratic separable extension of $F$, while $F(j)$ is inseparable.

\begin{ex}\label{anisot-alg+orthogonal-deg8-ind2-I^3=0}
    Let $k$ be a field of characteristic $2$. Set $k'=k(X_1,Y_1,X_2,Y_2)$. We consider the $k'$-quaternion algebras defined by 
    \[Q_1=[X_1+X_2,Y_1)_{k'}, Q_2=[X_1+X_2,Y_2)_{k'},\]
    \[Q_3=[X_1,Y_1Y_2)_{k'}
    \mbox{ and }Q_4=[X_2,Y_1Y_2)_{k'}.\]
    Set $Q=[X_1+X_2,Y_1Y_2)_{k'}$. Then $Q$ is a quaternion division algebra over $k'$, and we have the following Brauer equivalences:
    $$Q_1\otimes_{k'}Q_2\sim Q\sim Q_3\otimes_{k'}Q_4.$$
    For $1\leq i\leq4$, let $\gamma_i$ denote the canonical involution on $Q_i$.
    For all $1 \leq i,j\leq4$ with $i\neq j$, denote by $f_{ij}$ the canonical semi-trace on $\Sym(Q_i\otimes_FQ_j,\gamma_{i}\otimes\gamma_{j})$ as described in \cite[Notation 2.6]{DQM2021}. Consider 
    $$(A,\sigma,f)\in(Q_1\otimes_FQ_2,\gamma_{1}\otimes\gamma_{2},f_{12})\boxplus(Q_3\otimes_FQ_4,\gamma_{3}\otimes\gamma_{4},f_{34}).$$
    Then $(A,\sigma,f)$ is a $k'$-algebra with quadratic pair such that $\deg A=8$ and $A\sim Q$, and $\disc\sigma$ is trivial. 
    Let $\C^+$ and $\C^-$ denote the simple components of the Clifford algebra $\C(A,\sigma,f)$. 
    Then $\C^+\otimes_{k'}\C^-\sim Q$.
    By \cite[Proposition 5.4]{DQM2021}, we have, up to permutation, that 
    $$(\C^+,\sigma^+,f^+)\in(Q_1\otimes_FQ_3,\gamma_{1}\otimes\gamma_{3},f_{13})\boxplus(Q_2\otimes_FQ_4,\gamma_{2}\otimes\gamma_{4},f_{24}),$$
    $$(\C^-,\sigma^-,f^-)\in(Q_1\otimes_FQ_4,\gamma_{1}\otimes\gamma_{4},f_{14})\boxplus(Q_2\otimes_FQ_3,\gamma_{2}\otimes\gamma_{3},f_{23}).$$
    Set $B^+=[X_1,Y_2)_{k'}\otimes_{k'}[X_2,Y_1)_{k'}$, $B^-=[X_1,Y_1)_{k'}\otimes_{k'}[X_2,Y_2)_{k'}$.
    Then $B^+$ and $B^-$ are division algebras over $k'$, and we have that $\C^+\sim B^+$, $\C^-\sim B^-$, and $B^+\otimes_{k'}B^-\sim Q$.
    Let $q_{B^+}$ be an Albert form of $B^+$ and $k''=k'(q_{B^+})$ its function field. 
    Then $\ind B^+_{k''}=\ind Q_{k''}=2$ and $\ind B^-_{k''}=4$.

    Merkurjev's construction in \cite{Merkur1991u6} allows us to obtain a field extension $F/k'$ such that $\I_q^3F=0$, $\ind Q_{F}=2$ and $\ind B^+_{F}=\ind B^-_{F}=4$. Hence, the trialitarian triple 
\begin{equation*}
    \mbox{$T=((A,\sigma,f)_F,(\C^+,\sigma^+,f^+)_F,(\C^-,\sigma^-,f^-)_F)$}
\end{equation*}   
    has index $(2,4,4)$. 
    Similarly, there exists a field extension $F'/k''$ such that $\I^3_qF'=0$, $\ind Q_{F'}=\ind B^+_{F'}=2$ and $\ind B^-_{F'}=4$.
    Hence, the trialitarian triple 
\begin{equation*}
    \mbox{$T'=((A,\sigma,f)_{F'},(C^+,\sigma^+,f^+)_{F'},(C^-,\sigma^-,f^-)_{F'})$}
\end{equation*}    
    has index $(2,2,4)$. Both triples are anisotropic by \Cref{main.thm}.  
\end{ex}

\section*{Acknowledgments}
This work was supported by the Fonds Wetenschappelijk Onderzoek – Vlaanderen (FWO) in the FWO Odysseus Programme (project G0E6114N, \emph{Explicit Methods in Quadratic Form Theory}), by the FWO-Tournesol programme (project VS05018N), by the Fondazione Cariverona in the programme Ricerca Scientifica di Eccellenza 2018 (project \emph{Reducing complexity in algebra, logic, combinatorics - REDCOM}), and by T\"{U}B\.{I}TAK-221N171.

\end{document}